\documentclass[11pt]{amsart}
\usepackage{amsfonts}
\usepackage{ifthen}
\usepackage{amsthm}
\usepackage{amsmath}
\usepackage{graphicx}
\usepackage{amscd,amssymb,amsthm}
\usepackage{graphicx}
\usepackage{color}

\newcounter{minutes}
\divide\time by 60
\newcounter{hours}
\multiply\time by 60 \addtocounter{minutes}{-\time}

\title[Tur\'an type inequalities for generalized inverse trigonometric functions]{Tur\'an type inequalities for generalized inverse trigonometric functions}

\author[\'A. Baricz]{\'Arp\'ad Baricz$^{\star}$}
\address{Department of Economics, Babe\c{s}-Bolyai University, 400591 Cluj-Napoca, Romania}
\address{Institute of Applied Mathematics, John von Neumann Faculty of Informatics, \'Obuda University, 1034 Budapest, Hungary}
\email{bariczocsi@yahoo.com}

\author[B.A. Bhayo]{Barkat Ali Bhayo}
\address{Department of Mathematical Information Technology, University of Jyv\"askyl\"a, 40014 Jyv\"askyl\"a, Finland} \email{bhayo.barkat@gmail.com}

\author[M. Vuorinen]{Matti Vuorinen}
\address{Department of Mathematics and Statistics, University of Turku, 20014 Turku, Finland} \email{vuorinen@utu.fi}

\thanks{$^{\star}$The work of \'A. Baricz was supported by a research grant of the Babe\c{s}-Bolyai University for young researchers.}
\newtheorem{lemma}{Lemma}
\newtheorem{theorem}{Theorem}
\newtheorem{corollary}{Corollary}
\newtheorem{conjecture}{Conjecture}

\keywords{Eigenfunctions of $p$-Laplacian, generalized trigonometric function, log-convexity, log-concavity, Tur\'an-type inequalities, completely monotone functions, Bernstein functions.} \subjclass[2010]{33C99, 33B99}
\begin{document}

\def\thefootnote{}
\footnotetext{ \texttt{File:~\jobname .tex,
          printed: \number\year-\number\month-\number\day,
          \thehours.\ifnum\theminutes<10{0}\fi\theminutes}
} \makeatletter\def\thefootnote{\@arabic\c@footnote}\makeatother

\maketitle

%=======================================================================================================================================================

\begin{abstract}
In this paper we study the inverse of the eigenfunction $\sin_p$ of the one-dimensional
$p$-Laplace operator and its dependence on the parameter $p$, and we present a Tur\'an type
inequality for this function. Similar inequalities are given also
for other generalized inverse trigonometric and hyperbolic functions. In particular,
we deduce a Tur\'an type inequality for a series considered by Ramanujan, involving the digamma function.
\end{abstract}

%%%%%%%%%%%%%%%%%%%%%%%%%%%%%%%%%%%%%%%%%%%%%%%%%%%%%%%%%%%%%%%%%
%%%%%%%%%%%%%%%%%%%%%%%%%%% Section 1 %%%%%%%%%%%%%%%%%%%%%%%%%%%
%%%%%%%%%%%%%%%%%%%%%%%%%%%%%%%%%%%%%%%%%%%%%%%%%%%%%%%%%%%%%%%%%

\section{Introduction}

P. Lindqvist \cite{lp} studied the eigenfunction $\sin_p$ in connection with unidimensional nonlinear Dirichlet eigenvalue problem
for $p$-Laplacian. This function has become a standard tool in the analysis of more complicated equations with various applications e.g, see
\cite{benn,berg,bind,del,pem,eber,lpe}.

Motivated by the work of Lindqvist, several authors have studies on the equalities and inequalities of
the generalized trigonometric functions e.g, see
\cite{bhayo,bv2,be,egl,jiang,take,wang} and their bibliography.
Motivated by the many results on these generalized trigonometric functions, in this paper we make a contribution to the subject by showing some convexity properties \cite{avv2,baricz} and Tur\'an type inequalities for the inverse generalized trigonometric functions. These kind of inequalities are named after the Hungarian mathematician Paul Tur\'an who proved a similar inequality for Legendre polynomials. For more details on Tur\'an type inequalities we refer to the papers on hypergeometric functions \cite{turan1,turan2,turan3,turan4} and to the references therein. We deduce also a Tur\'an type inequality for a series involving the digamma function, which was considered by Ramanujan \cite{raman}. The monotonicity of the function $\pi_{p,q}$ is given in \cite{egl}, here we prove that the function $\pi_{p,q}$ is strictly geometrically convex and log-convex. We note that this study gives us new bounds for elementary functions in terms of generalized trigonometric and hyperbolic functions. We also mention that the results of this paper complements
the known Tur\'an type inequalities for Gaussian hypergeometric functions, see for example \cite{turan2,turan4}.

For the formulation of our main results we give first the following definitions of some classical functions, such as
\emph{gamma function} $\Gamma$, the $\emph
{ psi function}$ $\psi$ and the \emph{beta function} $B(\cdot,\cdot)$.
For $x>0$, $y>0$, these functions are defined by
$$\Gamma(x)=\int^\infty_0 e^{-t}t^{x-1}\,dt,\,\,\psi(x)=\frac{\Gamma'(x)}{\Gamma(x)},\,\,
B(x,y)=\frac{\Gamma(x)\Gamma(y)}{\Gamma(x+y)}.$$

For the given complex numbers $a,b$ and $c$ with $c\neq0,-1,-2,\ldots$,
the \emph{Gaussian hypergeometric function} is the
analytic continuation to the slit place $\mathbf{C}\setminus[1,\infty)$ of the series
$$F(a,b;c;z)={}_2F_1(a,b;c;z)=\sum^\infty_{n=0}\frac{(a,n)(b,n)}
{(c,n)}\frac{z^n}{n!},\qquad |z|<1.$$
Here $(a,0)=1$ for $a\neq 0$, and $(a,n)$ for $n\in\mathbb{N}$ is the \emph{shifted factorial}
 or the \emph{Appell symbol}
$$(a,n)=a(a+1)(a+2)\cdots(a+n-1).$$

For $x\in (0,1)$ and $p>0$ the generalized inverse trigonometric functions are defined as follows
$$\arcsin_p(x)=\int^x_0(1-t^p)^{-1/p}dt=x\,F\left(\frac{1}{p},\frac{1}{p};1+\frac{1}{p};x^p\right),$$
$$\arctan_p(x)=\int^x_0(1+t^p)^{-1}dt=x F\left(1,\frac{1}{p};1+\frac{1}{p};-x^p\right),$$
$${\rm arcsinh}_p(x)=\int^x_0(1+t^p)^{-1/p}dt=
xF\left(\frac{1}{p}\,,\frac{1}{p};1+\frac{1}{p};-x^p\right),$$
$${\rm arctanh}_p(x)=\int^x_0(1-t^p)^{-1}dt=xF\left
(1\,,\frac{1}{p};1+\frac{1}{p};x^p\right),$$
and $\arccos_p(x)=\arcsin_p((1-x^p)^{1/p}).$
We note that the eigenvalue problem \cite{dm}, $1<p<\infty$
$$-\Delta_p u=-\left(|u'|^{p-2}u'\right)'
=\lambda|u|^{p-2}u,\,u(0)=u(1)=0,$$
has eigenvalues
$\lambda_n=(p-1)(n \pi_p)^p,\,$
and eigenfunctions
$t\mapsto \sin_p(n \pi_p\, t),\,n\in\mathbb{N},\,$
where $\sin_p$ is the inverse function of $\arcsin_p$  and
$$\pi_p=\frac{2}{p}\int^1_0(1-s)^{-1/p}s^{1/p-1}ds=\frac{2}
{p}\,B\left(1-\frac{1}{p},\frac{1}{p}\right)=\frac{2 \pi}{p\,\sin\left(\frac{\pi}{p}\right)}\,,$$
see \cite{pem}. The other generalized trigonometric and hyperbolic functions $\cos_p:(0,a_p)\to(0,1),$ $\tan_p:(0,b_p)\to (0,1),$ $\sinh_p:(0,\infty)\to (0,1),$ and $\tanh_p:(0,\infty)\to (0,1)$ are defined as the inverse of the generalized inverse trigonometric and hyperbolic functions $\arccos_p,$ $\arctan_p,$ ${\rm arcsinh}_p$ and ${\rm arctanh}_p,$
where
$$a_p=\frac{\pi_p}{2},
\,b_p=\frac{1}{2p}\left(\psi\left(\frac{1+p}{2p}\right)
-\psi\left(\frac{1}{2p}\right)\right)=2^{-1/p}
F\left(\frac{1}{p},\frac{1}{p};1+\frac{1}{p};\frac{1}{2}\right).$$
We also consider for $p,q>0$ and $x\in(0,1)$ the generalized inverse trigonometric functions
$$\arcsin_{p,q}(x)=\int^x_0(1-t^q)^{-1/p}dt=x\,F\left(\frac{1}{p},\frac{1}{q};1+\frac{1}{q};x^q\right),$$
$${\rm arcsinh}_{p,q}(x)=\int^x_0(1+t^q)^{-1/p}dt=xF\left(\frac{1}{p}\,,\frac{1}{q};1+\frac{1}{q};-x^q\right),$$
which for $p=q$ reduces to $\arcsin_p(x)$ and ${\rm arcsinh}_p(x),$ also we denote $\arcsin_{p,q}(1)=\pi_{p,q}/2$.

Before we present the main results of this paper we recall some definitions, which will be used in the sequel. A function $f \colon (0,\infty)\to(0,\infty)$ is said to be logarithmically convex, or simply log-convex, if its natural logarithm $\ln f$ is convex, that is, for all $x,y>0$ and $\lambda\in[0,1]$ we have
   $$f(\lambda x+(1-\lambda)y) \leq \left[f(x)\right]^{\lambda}\left[f(y)\right]^{1-\lambda}.$$
The function $f$ is log-concave if the above inequality is reversed.  By definition, a function $g \colon (0,\infty)\rightarrow(0,\infty)$ is said to be geometrically (or multiplicatively)
convex if it is convex with respect to the geometric mean, that is, if for all $x,y>0$ and all $\lambda\in[0,1]$ the inequality
   $$g(x^{\lambda}y^{1-\lambda}) \leq[g(x)]^{\lambda}[g(y)]^{1-\lambda}$$
holds. The function $g$ is called geometrically concave if the above inequality is reversed. Observe that the
geometrical convexity of a function $g$ means that the function $\ln g$ is a convex function of $\ln x$ in
the usual sense. We also note that the differentiable function $f$ is log-convex (log-concave) if and only if
$x \mapsto f'(x)/f(x)$ is increasing (decreasing), while the differentiable function $g$ is geometrically convex (concave) if
and only if the function $x \mapsto xg'(x)/g(x)$ is increasing (decreasing), for more details see
\cite{baricz}. Finally, we recall that a function $h \colon (0,\infty)\rightarrow\mathbb{R}$ is said to be completely monotonic if $h$ has derivatives of all orders and satisfies $$(-1)^mh^{(m)}(x)\geq 0$$ for all $x>0$ and $m\in\{0,1,2,\dots\}.$ For properties of completely monotonic functions we refer to the paper \cite{samko} and to the references therein.

%%%%%%%%%%%%%%%%%%%%%%%%%%%%%%%%%%%%%%%%%%%%%%%%%%%%%%%%%%%%%%%%%
%%%%%%%%%%%%%%%%%%%%%%%%%%% Section 2 %%%%%%%%%%%%%%%%%%%%%%%%%%%
%%%%%%%%%%%%%%%%%%%%%%%%%%%%%%%%%%%%%%%%%%%%%%%%%%%%%%%%%%%%%%%%%

\section{Main results}
\setcounter{equation}{0}

Our first main result reads as follows.

\begin{theorem}\label{thm1}
For all $x\in(0,1)$ fixed, the following hold:
\begin{enumerate}
\item The functions $p\mapsto \arcsin_p(x)$ and
$p\mapsto {\rm arctanh}_p(x)$ are strictly completely monotonic and log-convex on $(0,\infty)$ .
Moreover, $p\mapsto \arcsin_p(x)$ is strictly geometrically convex on $(0,\infty).$
\item The
function $p\mapsto \arctan_p(x)$ is strictly increasing and concave on $(0,\infty)$.
\end{enumerate}
In particular, the following Tur\'an type inequalities are valid for all $p>1$ and $x\in(0,1)$
$$\arcsin_p^2(x)<\arcsin_{p-1}(x)\arcsin_{p+1}(x),$$
$${\rm arctanh}_p^2(x)<{\rm arctanh}_{p-1}(x){\rm arctanh}_{p+1}(x),$$
$$\arctan_p^2(x)>\arctan_{p-1}(x)\arctan_{p+1}(x).$$
\end{theorem}

The next corollary follows from Theorem \ref{thm1}.

\begin{corollary}
For $x\in (0,1)$, we have
$$\frac{\arcsin_3^2 (x)}{\arcsin_4 (x)}< \arcsin(x),\ \ \ \frac{{\rm arctanh}_3^2 (x)}{{\rm arctanh}_4 (x)}< {\rm arctanh}(x),\ \ \ \frac{\arctan_3^2 (x)}{\arctan_4 (x)}> \arctan(x),$$
and each inequality is sharp as $x\to0$.
\end{corollary}

The proof of the following result follows from Theorem \ref{thm1} and Lemma \ref{neu}.

\begin{corollary} For $p>0,$ $a\geq 1$ and $x\in (0,1)$, the following inequalities hold
$${\rm arcsin}_p(x)\leq \displaystyle\left(\frac{{\rm arcsin}_{2a}(x){\rm arcsin}_p(x)^a}
{{\rm arcsin}_{ap}(x)}\right)^{1/a}\leq {\rm arcsin}(x),$$
$${\rm arctanh}_p(x)\leq \displaystyle\left(\frac{{\rm arctanh}_{2a}(x){\rm arctanh}_p(x)^a}
{{\rm arctanh}_{ap}(x)}\right)^{1/a}\leq {\rm arctanh}(x),$$
with equality when $p=2$ and $a=1$.
\end{corollary}

Based on computer experiments we believe that the following results are true.

\begin{conjecture}
For $x\in(0,1)$ fixed, the function $p\mapsto {\rm arcsinh}_p(x)$ is strictly concave on $(0,\infty).$ In particular, the following Tur\'an type inequality is valid for all $p>1$ and $x\in(0,1)$
$${\rm arcsinh_p}^2(x)>{\rm arcsinh}_{p-1}(x){\rm arcsinh}_{p+1}(x).$$
\end{conjecture}

\begin{conjecture}
The following Tur\'an type inequalities hold for all $p>1$ and $x\in(0,1)$
$$\sin_p^2(x)>\sin_{p-1}(x)\sin_{p+1}(x),$$
$$\cos_p^2(x)>\cos_{p-1}(x)\cos_{p+1}(x),$$
$$\tan_p^2(x)<\tan_{p-1}(x)\tan_{p+1}(x),$$
$$\sinh_p^2(x)<\sinh_{p-1}(x)\sinh_{p+1}(x),$$
$$\tanh_p^2(x)>\tanh_{p-1}(x)\tanh_{p+1}(x).$$
\end{conjecture}

Now, we focus on the $\arcsin_{p,q}$ and ${\rm arcsinh}_{p,q}$ functions.

\begin{theorem}\label{thm2}
For all $x\in(0,1)$ fixed, the following hold:
\begin{enumerate}
\item $p\mapsto \arcsin_{p,q}(x)$ is completely monotonic and log-convex on $(0,\infty)$ for $q>0.$
\item $p\mapsto \arcsin_{p,q}(x)$ is strictly geometrically convex on $(0,\infty)$ for $q>0.$
\item $q\mapsto \arcsin_{p,q}(x)$ is completely monotonic and log-convex on $(0,\infty)$ for $p>0.$
\item $p\mapsto {\rm arcsinh}_{p,q}(x)$ is strictly increasing on $(0,\infty)$ and concave on $(\log\sqrt{2},\infty)$ for $q>0.$
\item $q\mapsto {\rm arcsinh}_{p,q}(x)$ is strictly increasing on $(0,\infty)$ for $p>0$ and strictly concave on $(0,\infty)$ for $p>1.$
\end{enumerate}
In particular, the following Tur\'an type inequalities are valid for $x\in(0,1)$
$$\arcsin_{p,q}^2(x)<\arcsin_{p-1,q}(x)\arcsin_{p+1,q}(x),\ \ \ \ p>1,q>0,$$
$${\rm arcsinh}_{p,q}^2(x)>{\rm arcsinh}_{p-1,q}(x){\rm arcsinh}_{p+1,q}(x), \ \ \ \ p>\log\sqrt{2}+1,q>0.$$
Moreover, for $x\in(0,1)$ we have the next Tur\'an type inequalities
$$\arcsin_{p,q}^2(x)<\arcsin_{p,q-1}(x)\arcsin_{p,q+1}(x),\ \ \ \ p>0,q>1,$$
$${\rm arcsinh}_{p,q}^2(x)>{\rm arcsinh}_{p,q-1}(x){\rm arcsinh}_{p,q+1}(x), \ \ \ \ p>1,q>1.$$
\end{theorem}

The next corollary follows from Theorem \ref{thm2}.

\begin{corollary}
For $x\in (0,1)$ and $p>0$, $p>\log\sqrt{2}$ respectively, we have
$$\frac{\arcsin_{p+1,p}^2 (x)}{\arcsin_{p+2,p}(x)}< \arcsin_p(x),\ \ \ {\rm arcsinh}_p(x)<\frac{{\rm arcsinh}_{p+1,p}^2(x)}{{\rm arcsinh}_{p+2,p}(x)},$$
and both of inequalities are sharp as $x\to0$.
\end{corollary}

\section{Concluding remarks and further results}
\setcounter{equation}{0}

{\bf A.} We would like to mention that it is possible to prove that $p\mapsto \arcsin_px$ is strictly decreasing and log-convex by using the hypergeometric series representation. Namely, it can be shown that for $x\in(0,1)$ fixed the function $a\mapsto F(a,a;a+1;x^{{1}/{a}})$ is strictly increasing on $(0,\infty)$. For this, first observe that
$$F(a,a;a+1;x^{{1}/{a}})=\sum_{n\geq 0}\varphi_n(a), \ \ \mbox{where}\ \ \ \varphi_n(a)=\frac{a(a,n)}{a+n}\frac{x^{{n}/{a}}}{n!}.$$
Taking the logarithm of $\varphi_n(a)$, we get
$$\log(\varphi_n(a))=\log(a)+\log(\Gamma(a+n))-\log(\Gamma(a))-\log(a+n)-\log(n!)+\frac{n}{a}\log(x).$$
Differentiating $\log(\varphi_n(a))$ with respect to $a$ we get
$$\frac{\varphi'_n(a)}{\varphi_n(a)}=\frac{1}{a}-\frac{1}{a+n}+\psi(a+n)-\psi(a)-\frac{n}{a^2}\log(x),$$
which is clearly strictly positive for all $a>0$ and $x\in(0,1)$. Here we used tacitly that the gamma function $\Gamma$ is log-convex,
 that is, the digamma function $\psi$ is increasing. This implies that
 $a\mapsto \varphi_n(a)$ is strictly increasing on $(0,\infty)$ for each $n\in\{1,2,\dots\}$ and $x\in(0,1)$ fixed. Consequently,
for $x\in(0,1)$ fixed the function $a\mapsto F(a,a;a+1;x^{{1}/{a}})$ is strictly increasing on $(0,\infty),$ as the infinite series of increasing functions. This in turn implies that indeed $p\mapsto \arcsin_px$ is strictly decreasing
on $(0,\infty)$ for all $x\in(0,1).$ Moreover, since the digamma function is concave,
it follows that for all $a>0,$ $x\in(0,1)$ and $n\in\{1,2,\dots\}$ we have
$$\left[\frac{\varphi'_n(a)}{\varphi_n(a)}\right]'=-\frac{1}{a^2}+\frac{1}{(a+n)^2}+\psi'(a+n)-\psi'(a)+\frac{2n}{a^3}\log(x)<0,$$
which means that $a\mapsto \varphi_n(a)$ is strictly log-concave on $(0,\infty)$ for each $n\in\{1,2,\dots\}$ and $x\in(0,1)$ fixed. Now, since for $n\in\{1,2,\dots\},$ $p>0$ one has
$$(\log\varphi_n(1/p))''=-\frac{1}{p^2}\cdot\left(\frac{\varphi_n'(1/p)}{\varphi_n(1/p)}\right)'+\frac{2}{p^3}\cdot\frac{\varphi_n'(1/p)}{\varphi_n(1/p)}>0,$$ we get that $p\mapsto \varphi_n\left({1}/{p}\right)$ is strictly log-convex on $(0,\infty)$ for each $n\in\{1,2,\dots\}$ and $x\in(0,1)$ fixed, and hence $p\mapsto \arcsin_px$ is indeed strictly log-convex on $(0,\infty)$ for all $x\in(0,1),$ as the infinite sum of strictly log-convex functions.

Now, recall that a continuous function $f:(0,\infty)\to(0,\infty)$ is a Bernstein function (see \cite{vondracek}) if $(-1)^k f^{(k)}(x)\leq0$ for $x>0$ and $k\in\{1,2,\dots\},$ that is $f'$ is a completely monotone function. We would like to mention that if $a>0,$ $n\in\{1,2,\dots\}$ and $x\in(0,1)$ are such that $\varphi_n(a)>1,$ then the function $a\mapsto \log(\varphi_n(a))$ is in fact a Bernstein function, that is, $(-1)^k \log(\varphi_n(a))^{(k)}\leq0$ for $a>0,$ $x\in(0,1)$ and $k\in\{1,2,\dots\}.$ Indeed, for $a>0,$ $x\in(0,1)$ and $k\in\{1,2,\dots\}$ we have
$$\log(\varphi_n(a))^{(k)}=(-1)^{k-1}(k-1)!\left[\frac{1}{a^k}-\frac{1}{(a+n)^k}\right]+\psi^{(k-1)}(a+n)-\psi^{(k-1)}(a)+\frac{(-1)^kk!n}{a^{k+1}}\log(x)$$
and consequently
$$\frac{(-1)^k\log(\varphi_n(a))^{(k)}}{(k-1)!}=-\left[\frac{1}{a^k}-\frac{1}{(a+n)^k}\right]+\sum_{m\geq0}\left[\frac{1}{(a+n+m)^k}-\frac{1}{(a+m)^k}\right]+\frac{kn}{a^{k+1}}\log(x)<0.$$
Here we used tacitly that
$$\psi^{(k-1)}(a+n)-\psi^{(k-1)}(a)=(-1)^{k}(k-1)!\sum_{m\geq0}\left[\frac{1}{(a+n+m)^k}-\frac{1}{(a+m)^k}\right].$$

Finally, we mention that a similar procedure to that mentioned above can be applied to prove that
$$p\mapsto {\rm arctanh}_p(x)=xF\left
(1\,,\frac{1}{p};1+\frac{1}{p};x^p\right)=x\sum_{n\geq0}\frac{x^{pn}}{pn+1}$$
is strictly decreasing and log-convex on $(0,\infty)$ for all $x\in(0,1)$ fixed.

{\bf B.} In the first main theorem we mentioned that the function $p\mapsto \arcsin_p(x)$ is strictly geometrically convex on $(0,\infty)$ for $x\in(0,1),$ and in the proof we used Lemma \ref{geo}. We note that the origins of such kind of results goes back to Montel. More precisely, Montel \cite{montel} proved the following result: if the function $f:(0,a)\to (0,\infty)$ is geometrically convex, then the function $x\mapsto \displaystyle\int_0^xf(t)dt$ is also geometrically convex on $(0,a).$ Moreover, it is known (see \cite{baricz,zhang}) that the above result remains true if we replace the word ``convex'' with ``concave''. Now, consider the functions $f,g,r,s:(0,1)\to(0,\infty),$ defined by
$$f(t)=(1-t^p)^{-1/p},\ \ g(t)=(1+t^p)^{-1},\ \ r(t)=(1+t^p)^{-1/p}, \ \ s(t)=(1-t^p)^{-1}.$$
Then for all $t\in(0,1)$ and $p>0$ we have $$\left[\frac{ts'(t)}{s(t)}\right]'=p\left[\frac{tf'(t)}{f(t)}\right]'=\frac{p^2t^{p-1}}{(1-t^p)^2}>0$$
and
$$\left[\frac{tg'(t)}{g(t)}\right]'=p\left[\frac{tr'(t)}{r(t)}\right]'=-\frac{p^2t^{p-1}}{(1+t^p)^2}<0.$$
Combining these with the above results it follows that for $p>0$ the functions $x\mapsto \arcsin_p x$ and $x\mapsto {\rm arctanh}_p\ x$ are strictly geometrically convex on $(0,1),$ while the functions $x\mapsto {\rm arcsinh}_p\ x$ and $x\mapsto \arctan_p x$ are strictly geometrically concave on $(0,1).$ These results for $p>1$ were proved recently in \cite{bhayo} by using a different approach.

{\bf C.} Observe that
$$a_p=\frac{\pi_p}{2}=\arcsin_p(1)=\int_0^1(1-t^p)^{-1/p}dt=\frac{1}{p}\int_0^1(1-s)^{-1/p}s^{1/p-1}ds.$$
Now, since by Theorem \ref{thm1} we have that $p\mapsto \arcsin_p(x)$ is strictly geometrically convex and log-convex on $(0,\infty)$ for $x\in (0,1),$ by letting $x$ tend to $1$ we get that $p\mapsto \pi_p=2\arcsin_p(1)$ is strictly geometrically convex and log-convex on $(0,\infty).$ Consequently, for $\alpha\in(0,1)$ and $p,q>0$ such that $p\neq q$ we have
$$\pi_{p^{\alpha}q^{1-\alpha}}<\pi_p^{\alpha}\pi_q^{1-\alpha}\ \ \ \ \mbox{and}\ \ \ \ \pi_{{\alpha}p+{(1-\alpha)}q}<\pi_p^{\alpha}\pi_q^{1-\alpha}.$$
The first inequality implies that we have also
$$\pi_{p^{\alpha}q^{1-\alpha}}<{\alpha}\pi_p+(1-{\alpha})\pi_q.$$
Observe that for $\alpha=\frac{1}{2},$ $p=s-1$ and $q=s+1$ the second inequality reduces to the next Tur\'an type inequality for $s>1$
$$\pi_s^2<\pi_{s-1}\pi_{s+1}$$ which is equivalent to
$$\frac{\sin^2\frac{\pi}{s}}{\sin\frac{\pi}{s-1}\sin\frac{\pi}{s+1}}>\frac{s^2}{(s-1)(s+1)}.$$
We also mention that the first inequality of this remark for $\alpha=\frac{1}{2}$ and the third inequality of this remark for $\alpha\in(0,1)$ were proved recently by Bhayo and Vuorinen \cite{bhayo}.

{\bf D.} Now, we focus on $b_p.$ Observe that
$$b_p=\arctan_p(1)=\int_0^1(1+t^p)^{-1}dt.$$ Now, taking into account that according to Theorem \ref{thm1} the function $p\mapsto \arctan_p x$ is strictly concave on $(0,\infty)$ for $x\in(0,1),$ by tending with $x$ to $1$ we obtain that $p\mapsto b_p$ is also concave on $(0,\infty).$ In particular we have
$$b_{\alpha p+(1-\alpha)q}>\alpha b_p+(1-\alpha)b_q>(b_p)^{\alpha}(b_q)^{1-\alpha}$$
for all $\alpha\in(0,1)$ and $p,q>0$ such that $p\neq q.$ Choosing $\alpha=\frac{1}{2}$ and $p=s-1,$ $q=s+1,$ we obtain for $s>1$ the next Tur\'an type inequality
$$b_s^2>b_{s-1}b_{s+1}.$$
We note that the series $b_p$ was considered by Ramanujan \cite[p. 184-190]{raman} and for $p\in\{3,4,5,6,8,10\}$ its values were computed. For example, we have
$$b_2=\arctan_2(1)=\frac{\pi}{4},\ b_3=\arctan_3(1)=\frac{1}{3}\log 2+\frac{\pi}{3\sqrt{3}},\ b_4=\arctan_4(1)=\frac{\pi}{4\sqrt{2}}-\frac{\log\left(\sqrt{2}-1\right)}{2\sqrt{2}},$$
and hence the above Tur\'an type inequality for $s=3$ becomes
$$0.6983089976{\dots}=\left(\frac{1}{3}\log 2+\frac{\pi}{3\sqrt{3}}\right)^2>\frac{\pi^2}{16\sqrt{2}}-\frac{\pi}{4}\cdot\frac{\log\left(\sqrt{2}-1\right)}{2\sqrt{2}}=0.6809189919{\dots}.$$

{\bf E.} Finally, we consider the expression
$$\frac{\pi_{p,q}}{2}=\arcsin_{p,q}(1)=\int_0^1(1-t^q)^{-1/p}dt=\frac{2}{q}B\left(1-\frac{1}{p},\frac{1}{q}\right).$$
Recall that Theorem \ref{thm2} asserts that $p\mapsto \arcsin_{p,q}(x)$ is strictly geometrically convex and log-convex on $(0,\infty)$ for $x\in (0,1)$ and $q>0.$ By tending with $x$ to $1$ we get that $p\mapsto \pi_{p,q}$ is strictly geometrically convex and log-convex on $(0,\infty)$ for $q>0.$ Consequently, for $\alpha\in(0,1)$ and $p_1,p_2,q>0$ such that $p_1\neq p_2$ we have
$$\pi_{p_1^{\alpha}p_2^{1-\alpha},q}<\pi_{p_1,q}^{\alpha}\pi_{p_2,q}^{1-\alpha}\ \ \ \ \mbox{and}\ \ \ \ \pi_{{\alpha}p_1+{(1-\alpha)}p_2,q}<\pi_{p_1,q}^{\alpha}\pi_{p_2,q}^{1-\alpha}.$$
The first inequality implies that we have also
$$\pi_{p_1^{\alpha}p_2^{1-\alpha},q}<{\alpha}\pi_{p_1,q}+(1-{\alpha})\pi_{p_2,q}.$$
Observe that for $\alpha=\frac{1}{2},$ $p_1=s-1$ and $p_2=s+1$ the second inequality reduces to the next Tur\'an type inequality for $s>1$ and $q>0$
$$\pi_{s,q}^2<\pi_{s-1,q}\pi_{s+1,q}.$$
These results extend the results from remark {\bf C}. We also mention that by means of Theorem \ref{thm2} the function $q\mapsto \pi_{p,q}$ is strictly log-convex on $(0,\infty)$ for $p>0$ and the next Tur\'an type inequality is valid for $s>1$ and $p>0$
$$\pi_{p,s}^2<\pi_{p,s-1}\pi_{p,s+1}.$$

\section{Lemmas and proofs of the main results}
\setcounter{equation}{0}

In this section our aim is to present the proofs of the main results together with the preliminary results which we use in the proofs.

\begin{lemma}\label{neu}\cite[Thm 2.1]{ne2} Let $f:(0,\infty)\to (0,\infty)$
be a differentiable, log-convex function and let $a\geq 1$. Then $g(x)=(f(x))^a/f(a\,x)$
 decreases on its domain. In particular, if $0\leq x\leq y\,,$ then the following inequalities
 $$\frac{(f(y))^a}{f(a\,y)}\leq\frac{(f(x))^a}{f(a\,x)}\leq (f(0))^{a-1}$$
 hold true. If $0<a\leq 1$, then the function $g$ is an increasing function on $(0,\infty)$
and inequalities are reversed.
\end{lemma}

\begin{lemma}\label{geo}
Let $b>a>0.$ If the positive function $\nu\mapsto K(\nu,t)$ is (strictly) geometrically convex on $[a,b]$ for $t\in[0,x]$
with $x>0$, then the function
$$\nu\mapsto F_{\nu}(x)=\int_0^xK(\nu,t)dt$$
is also (strictly) geometrically convex on $[a,b]$.
\end{lemma}

\begin{proof}[\bf Proof of Lemma \ref{geo}]
The result follows immediately from the well-known H\"older-Rogers inequality for integrals. Namely, we have
\begin{align*}
F_{\nu^{\alpha}\mu^{1-\alpha}}(x)&=\int_0^x K\left(\nu^{\alpha}\mu^{1-\alpha},t\right)dt\leq \int_0^xK^{\alpha}(\nu,t)K^{1-\alpha}(\mu,t)dt\\&
\leq \left[\int_0^xK(\nu,t)dt\right]^{\alpha}\left[\int_0^xK(\mu,t)dt\right]^{1-\alpha}=F_{\nu}^{\alpha}(x)F_{\mu}^{1-\alpha}(x),
\end{align*}
where $\nu,\mu\in[a,b]$ and $\alpha\in[0,1].$
\end{proof}

\begin{proof}[\bf Proof of Theorem \ref{thm1}]
For the proof of part (1), let $t\in(0,x),\,x\in(0,1)$ be fixed. Let us consider the function $f:(0,\infty)\to \mathbb{R},$ defined by
$$f(p)=\log\left(1-t^p\right)^{-1/p}=-\frac{1}{p}\log\left(1-t^p\right).$$
Observe that for $p>0$ and $t\in(0,1)$
$$f'(p)=\frac{1}{p}\frac{t^p}{1-t^p}\log t+\frac{1}{p^2}\log\left(1-t^p\right)<0,$$
$$f''(p)=-\frac{2}{p^2}\frac{t^p}{1-t^p}\log t+\frac{t^p}{p}\left(\frac{\log t}{1-t^p}\right)^2-\frac{2}{p^3}\log\left(1-t^p\right)>0.$$
Consequently, the function $f$ is strictly decreasing and convex, which in turn implies that $p\mapsto \left(1-t^p\right)^{-1/p}$ is strictly decreasing and log-convex on $(0,\infty)$. In other words, the integrand of $\arcsin_px$ is strictly decreasing and log-convex on $(0,\infty)$.
Now, by using the fact that the integral preserves the monotonicity and log-convexity,
it follows that the function $p\mapsto \arcsin_px$ is strictly decreasing and log-convex
 on $(0,\infty)$. Now, observe that for $s(p)=e^{f(p)}$ we have
$$\frac{ps'(p)}{s(p)}=\frac{t^p}{1-t^p}\log t+\frac{1}{p}\log\left(1-t^p\right),$$
$$\left[\frac{ps'(p)}{s(p)}\right]'=-\frac{1}{p}\frac{t^p}{1-t^p}\log t+
t^p\left(\frac{\log t}{1-t^p}\right)^2-\frac{1}{p^2}\log\left(1-t^p\right)>0,$$
where $p>0$ and $t\in(0,1).$ This means that the integrand of $\arcsin_px$ is strictly geometrically convex on $(0,\infty)$. By
Lemma \ref{geo}, it follows that $p\mapsto \arcsin_px$ is strictly geometrically convex on $(0,\infty)$.

Now, for $t\in(0,1)$ fixed let us consider the function $g:(0,\infty)\to \mathbb{R},$ defined by
$$g(p)=\log\left(1-t^p\right)^{-1}=-\log\left(1-t^p\right).$$
We get
$$g'(p)=\frac{t^p}{1-t^p}\log t<0,\ \ g''(p)={t^p}\left(\frac{\log t}{1-t^p}\right)^2>0,$$
and consequently $p\mapsto \left(1-t^p\right)^{-1}$ is strictly decreasing and log-convex on $(0,\infty)$. By using again the fact that the integral preserves the monotonicity and log-convexity, it follows that the function $p\mapsto {\rm arctanh}_px$ is strictly decreasing and log-convex on $(0,\infty)$ for all $x\in(0,1)$ fixed.

Thus, we proved that the functions $p\mapsto \arcsin_p(x)$ and $p\mapsto {\rm arctanh}_p(x)$ are indeed strictly decreasing and log-convex on $(0,\infty)$. Now, let us focus on the complete monotonicity. Recall (see \cite{samko}) that the composition of a completely monotonic function with a function whose derivative is completely monotone is also completely monotonic. This implies that for $t\in(0,1)$ the function $p\mapsto g(p)=-\log(1-t^p)$ is completely monotonic on $(0,\infty)$ since $p\mapsto -\log p$ is completely monotonic on $(0,1)$ and $p\mapsto -t^p\log t$ is completely monotonic on $(0,\infty).$ On the other-hand it is known that the product of completely monotonic functions is also completely monotonic (see \cite{samko}), which in turn implies that $p\mapsto f(p)=\frac{1}{p}\cdot\left(-\log(1-t^p)\right)$ is completely monotonic on $(0,\infty)$ for $t\in(0,1).$ We note that the positive function $\varphi$ is said to be logarithmically completely monotonic if it satisfies $(-1)^m\left[\log \varphi(x)\right]^{(m)}\geq 0$ for all $x>0$ and $m\in\{0,1,2,\dots\}.$ We also note that every logarithmically completely monotonic function is completely monotonic, and each completely monotonic function is log-convex, see \cite{berg} and \cite[p. 167]{widder}. The above results imply that the functions $p\mapsto (1-t^p)^{-1}$ and $p\mapsto (1-t^p)^{-1/p}$ are logarithmically completely monotonic, and hence completely monotonic on $(0,\infty)$ for $t\in (0,1).$ These show that indeed the integrands of ${\rm arctanh}_p\, x$ and $\arcsin_p x$ are strictly completely monotonic and log-convex as functions of $p$ on $(0,\infty)$ for $t\in(0,1),$ and also for $t\in(0,x).$ Thus, by using the property that the integral preserves the complete monotonicity (see \cite{samko}), we proved that the functions $p\mapsto \arcsin_p(x)$ and $p\mapsto {\rm arctanh}_p(x)$ are indeed strictly completely monotonic and hence log-convex on $(0,\infty)$.

For the proof of part (2), let us consider the function $h:(0,\infty)\to \mathbb{R},$ defined by
$h(p)=(1+t^p)^{-1}$,
for fixed $t\in(0,1)$.
We have
$$h'(p)=-\frac{t^p}{(1+t^p)^2}\log t>0\quad  \mbox{and} \quad
h''(p)=-\frac{t^p(1-t^p)}{(1+t^p)^3}(\log t)^2<0,$$
and consequently $h$ is strictly increasing and concave. Consequently, the function $p\mapsto \arctan_p(x)$ is strictly increasing on $(0,\infty)$ for all $x\in(0,1)$ fixed. Moreover for $t\in(0,1),$ $\alpha\in(0,1),$ $p,q>0$ such that $p\neq q$ we have
$$h(\alpha p+(1-\alpha)q)>\alpha h(p)+(1-\alpha)h(q),$$
and hence
\begin{align*}\arctan_{\alpha p+(1-\alpha)q}&(x)=\int_0^xh(\alpha p+(1-\alpha)q)dt\\&>\alpha \int_0^xh(p)dt+(1-\alpha)\int_0^xh(q)dt=\alpha \arctan_p(x)+(1-\alpha) \arctan_q(x),\end{align*}
which means that $p\mapsto \arctan_p(x)$ is strictly concave on $(0,\infty)$ for all $x\in(0,1)$ fixed. Now, since the concavity is stronger than the log-concavity, it follows that $p\mapsto \arctan_p(x)$ is strictly log-concave on $(0,\infty)$. This completes the proof.
\end{proof}

\begin{proof}[\bf Proof of Theorem \ref{thm2}] We consider the two-variable functions $f,g:(0,\infty)^2\to(0,\infty),$ defined
$$f(p,q)=(1-t^q)^{-1/p},\ \ g(p,q)=(1+t^q)^{-1/p},$$
where $t\in(0,1).$ Since $p\mapsto 1/p$ is completely monotonic on $(0,\infty),$ the function $p\mapsto \log f(p,q)=-\frac{1}{p}\log(1-t^q)$ for $q>0$ and $t\in(0,1)$ is completely monotonic on $(0,\infty),$ and consequently the function $p\mapsto f(p,q)$ for $q>0$ and $t\in(0,1)$ is also completely monotonic on $(0,\infty).$ This implies that the function $p\mapsto \arcsin_{p,q}(x)$ is completely monotonic, and hence log-convex on $(0,\infty)$ for $q>0$ and $x\in(0,1).$ According to the proof of Theorem \ref{thm1}, the function $q\mapsto \log f(p,q)$ for $p>0$ and $t\in(0,1)$ is completely monotonic on $(0,\infty),$ and consequently the function $q\mapsto f(p,q)$ for $p>0$ and $t\in(0,1)$ is also completely monotonic on $(0,\infty).$ This implies that the function $q\mapsto \arcsin_{p,q}(x)$ is completely monotonic, and hence log-convex on $(0,\infty)$ for $p>0$ and $x\in(0,1).$

Since for $q>0$ and $t\in(0,1)$ the function $$p\mapsto \frac{p\frac{\partial f(p,q)}{\partial p}}{f(p,q)}=\frac{1}{p}\log(1-t^q)$$
is strictly increasing on $(0,\infty),$ by Lemma \ref{geo} we obtain that $p\mapsto \arcsin_{p,q}(x)$ is strictly geometrically convex on $(0,\infty)$ for $q>0$ and $x\in(0,1).$

On the other hand, for $t\in(0,1)$ we have
$$\frac{\partial g(p,q)}{\partial p}=\frac{1}{p^2}(1+t^q)^{-1/p}\log(1+t^q)>0,\ \ \mbox{if}\ \ p,q>0,$$
$$\frac{\partial^2 g(p,q)}{\partial p^2}=\frac{1}{p^3}(1+t^q)^{-1/p}\left(\log(1+t^q)-2p\right)\log(1+t^q)<0,\ \ \mbox{if}\ \ p>\log\sqrt{2},q>0,$$
$$\frac{\partial g(p,q)}{\partial q}=-\frac{1}{p}(1+t^q)^{-\frac{1}{p}-1}t^q\log t>0,\ \ \mbox{if}\ \ p,q>0,$$
$$\frac{\partial^2 g(p,q)}{\partial q^2}=-\frac{1}{p^2}(1+t^q)^{-\frac{1}{p}-2}t^q(p-t^q)(\log t)^2<0,\ \ \mbox{if}\ \ p>1,q>0.$$
Consequently the integrand of ${\rm arcsinh}_{p,q}(x)$ is strictly decreasing on $(0,\infty)$ with respect to $p,$ and also with respect to $q,$ when $p,q>0.$ Moreover, the integrand of ${\rm arcsinh}_{p,q}(x)$ is strictly concave with respect to $p$ on $(\log\sqrt{2},\infty)$ for $q>0$ and $t\in(0,1);$ and is strictly concave with respect to $q$ on $(0,\infty)$ for $p>1$ and $t\in(0,1).$ Since the integral preserves the monotonicity and concavity, it follows that $p\mapsto {\rm arcsinh}_{p,q}(x)$ is strictly decreasing on $(0,\infty)$ and concave on $(\log\sqrt{2},\infty)$ for $q>0$ and $x\in(0,1);$ and $q\mapsto {\rm arcsinh}_{p,q}(x)$ is strictly decreasing and concave on $(0,\infty)$ for $x\in(0,1),$ $p>0,$ and $p>1$ respectively. Finally, since the concavity implies the log-concavity, the proof of this theorem is complete.
\end{proof}

%%%%%%%%%%%%%%%%%%%%%%%%%%%%%%%%%%%%%%%%%%%%%%%%%%%%%%%%%%%%%%%%%
%%%%%%%%%%%%%%%%%%%%%%%%%%% Biblography %%%%%%%%%%%%%%%%%%%%%%%%%%%
%%%%%%%%%%%%%%%%%%%%%%%%%%%%%%%%%%%%%%%%%%%%%%%%%%%%%%%%%%%%%%%%%

\end{document}